\documentclass[12pt]{amsart}

\usepackage{graphicx}
\usepackage{floatflt}

\newtheorem{theorem}{Theorem}

\newtheorem{lemma}[theorem]{Lemma}
\newtheorem{proposition}[theorem]{Proposition}

\theoremstyle{definition}

\newtheorem{question}{Question}

\theoremstyle{remark}
\newtheorem*{remark}{Remark}
\newtheorem*{acknowledgement}{Acknowledgements}

  {\end{list}}
\usepackage{amssymb}
\usepackage{amsmath}
\usepackage{verbatim}

\def\det{{\mathop{\rm det}}}

\def\BF1{{\mathbf 1}}

\def\NN{{\mathbb N}}
\def\RR{{\mathbb R}}

\theoremstyle{definition}
\theoremstyle{remark}

\def\NN{{\mathbb N}}
\newcommand{\thmref}[1]{Theorem~\ref{#1}}
\newcommand{\lemref}[1]{Lemma~\ref{#1}}
\newcommand{\eqnref}[1]{Equation~(\ref{#1})}

\newcommand{\propref}[1]{Proposition~\ref{#1}}
\newcommand{\figref}[1]{Figure~\ref{#1}}

\begin{document}
\title[volumes of $n$-simplices with vertices on a polynomial space curve]{volumes of $n$-simplices with vertices on a polynomial space curve}



\keywords{Complete Symmetric Polynomials, Disciminant, Alternant, Triangle Area, volumes of $n$-parallelpipeds, volumes of
$n$-simplices.}
\begin{abstract}
In this paper,
we give a formula for the area of the triangle formed by the
vertices that live on a given polynomial, and we generalize this formula to
the volumes of $n$-simplices with vertices on a polynomial space curve.
To prove these results, we use induction arguments and a well known identity
for complete symmetric polynomials.
\end{abstract}

\maketitle



\begin{floatingfigure}[r]{4 in}
\centering
\includegraphics[scale=0.65]{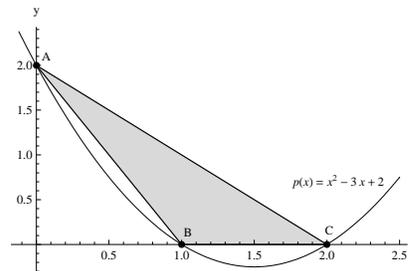} \caption{Triangle with vertices (0, 2), (1, 0) and (2, 0) living on the parabola $p(x)=x^2-3x+2$.} \label{fig curvewd}
\end{floatingfigure}
Given a number $x_1 \in \RR$ and a parabola $p(x)=x^2+bx+c$ defined over $\RR$,
triangles with vertices $(x_1,p(x_1))$, $(x_1+1,p(x_1+1))$ and $(x_1+2,p(x_1+2))$
have unit area whatever the values of $b$, $c$ and $x_1$. \figref{fig curvewd} shows an example.

Starting with this observation, we provide answers to the following questions:
\begin{question}\label{ques triangle}
Let $p(t)=\sum_{k=0}^{N} a_{k}t^k$ be polynomial of degree $N \geq 2$, i.e. $a_{N}\not=0$, with real coefficients.
For any given three real numbers $x_1$, $x_2$, $x_3$, we consider the triangle $ABC$ formed by the vertices
$\displaystyle A=(x_1 ,p(x_1)), \,  B=(x_2 ,p(x_2)), \,  C=(x_3 ,p(x_3)).$
If any two of $x_1, x_2, x_3$ are equal to each other then, clearly the area of $ABC$ is $0$.

Can we obtain a formula in factorized form for the area of $ABC$ so that we can see the effects of
the polynomial $p(t)$, and $x_1$, $x_2$, and $x_3$? If so, can we generalize the result to higher dimensions?

The answers we provide for this question are \thmref{Main Theorem} and \thmref{genmainthhm}.
\end{question}

\begin{question}\label{ques factor}
The following identity is well known.
\begin{equation}\label{eqn2}
x^m-y^m=(x-y)(x^{m-1}+x^{m-2}y+ \dots +xy^{m-2}+y^{m-1}), \quad \text{for any 2 $\leq m \in \NN$}.
\end{equation}
What further generalizations of this identity are there? For example, what would be the analogues identity in
$n \geq 3$ variables?

\propref{prop3var} and \thmref{genprop3var} answer question \ref{ques factor}.
\end{question}

\begin{question}\label{ques poly}
For any given $d \geq 0$ and $n \geq 2$, are there any ``short" algebraic expressions for the complete symmetric polynomial of
degree $d$ in $n$ variables?

Question \ref{ques poly} is closely related to question \ref{ques factor}. \thmref{thm symmetric poly} provides an answer for question \ref{ques poly}.
\end{question}

We start by considering question \ref{ques factor}. First,
we notice the appearance of all monomials of degree $m$, in $x$ and $y$, in \eqnref{eqn2}. This suggests that any
generalization of (\ref{eqn2}) to more variables should involve the complete symmetric polynomials.
Indeed this is the case, as we shall see. Let's recall
their definition and establish some notation.

For any $0\leq d \in \NN$ and $1\leq n \in \NN$, the complete symmetric polynomial of degree $d$ in $n$ variables is defined
to be the sum of all monomials of degree $d$ in $n$ variables. For variables $x_{1}, x_{2}, \ldots, x_{n}$, we set
$\displaystyle h_{d}(x_1, x_2, \ldots, x_n)= \sum_{d_1+d_2+\cdots+d_n=d}x_{1}^{d_1}x_{2}^{d_2} \ldots x_{n}^{d_n},$
where $d_i$ is a nonnegative integer for any $i=1, 2, \ldots, n$. By convention, we set $h_{0}(x_1, x_2, \dots, x_n)=1$ for any $n \geq 1$.
See \cite[Chapter 1]{MD} for a discussion of basic properties of complete symmetric polynomials.

We can rewrite \eqnref{eqn2} as $x_{1}^{m}-x_{2}^m=(x_{1}-x_{2})h_{m-1}(x_{1},x_{2})$ for $m \geq 1$.
\begin{lemma}\label{lemdech}
For $1 \leq s \in \NN$ and $2 \leq n \in \NN$,
$\displaystyle h_{s}(x_1,x_2, \dots, x_n)=\sum_{j=1}^{n}x_{j}h_{s-1}(x_j,x_{j+1}, \dots, x_n).$
\end{lemma}
\begin{proof}
The proof follows by grouping the monomials  as the ones that are divisible by $x_1$, the ones that are not divisible
by $x_1$ but divisible by $x_2$, the ones that are not divisible
by $x_1$ and $x_2$ but divisible by $x_3$, and so on.
\end{proof}
\begin{lemma}\label{lemh}
For $2 \leq m \in \NN$,
\begin{equation*}
h_{m}(x_1,x_2,x_3)=(x_1+x_2)h_{m-1}(x_1,x_2,x_3)-x_1x_2h_{m-2}(x_1,x_2,x_3)+x_{3}^m .
\end{equation*}
\end{lemma}
\begin{proof} We use \lemref{lemdech} to derive the result.
\begin{equation*}
\begin{split}
&h_{m}(x_1,x_2,x_3)
\\ &=x_{1}h_{m-1}(x_1,x_2,x_3)+x_2h_{m-1}(x_2,x_3)+x_{3}^m, \quad \text{by \lemref{lemdech} with $s=m$ and $n=3$}
\\ &=x_{1}h_{m-1}(x_1,x_2,x_3)+x_2\big(x_2h_{m-2}(x_2,x_3)+x_{3}^{m-1}\big)+x_{3}^m,  \quad \text{by \lemref{lemdech}
with $s=m-1$}
\\ &\text{and $n=2$. Then, \lemref{lemdech} with  $s=m-1$ and $n=3$ gives}
\\ &=x_{1}h_{m-1}(x_1,x_2,x_3)+x_2\big(h_{m-1}(x_1,x_2,x_3)-x_1h_{m-2}(x_1,x_2,x_3)
-x_{3}^{m-1}+x_{3}^{m-1} \big)+x_{3}^m
\\ &=(x_1+x_2)h_{m-1}(x_1,x_2,x_3)-x_1x_2h_{m-2}(x_1,x_2,x_3)+x_{3}^m .
\end{split}
\end{equation*}
\end{proof}
The following proposition gives a partial answer to question \ref{ques factor}.
It is the generalization of \eqnref{eqn2} to three variables.
\begin{proposition}\label{prop3var}
For $1 \leq m \in \NN$ and variables $x_1, x_2, x_3$,
\begin{equation*}
(x_{3}-x_{2})x_{1}^{m+1}+(x_{1}-x_{3})x_{2}^{m+1}+(x_{2}-x_{1})x_{3}^{m+1}
=(x_3-x_2)(x_3-x_1)(x_2-x_1)h_{m-1}(x_1,x_2,x_3) .
\end{equation*}
\end{proposition}
\begin{proof}
The proof is by induction on $m$.
Since $h_{0}(x_1,x_2,x_3)=1$ and $h_{1}(x_1,x_2,x_3)=x_1+x_2+x_3$,
the equalities for $m=1$ and $m=2$ can be verified by a simple algebra.

Suppose that the equality is true for all $i$ such that $2 \leq i \leq m$. Then, in particular,
for the cases $i=m$ and $i=m-1$ we have
\begin{equation}\label{eqn3}
(x_{3}-x_{2})x_{1}^{m+1}+(x_{1}-x_{3})x_{2}^{m+1}+(x_{2}-x_{1})x_{3}^{m+1}
=(x_3-x_2)(x_3-x_1)(x_2-x_1)h_{m-1}(x_1,x_2,x_3),
\end{equation}
\begin{equation}\label{eqn4}
(x_{3}-x_{2})x_{1}^{m}+(x_{1}-x_{3})x_{2}^{m}+(x_{2}-x_{1})x_{3}^{m}
=(x_3-x_2)(x_3-x_1)(x_2-x_1)h_{m-2}(x_1,x_2,x_3) .
\end{equation}
For $i=m+1$, by \lemref{lemh},
\begin{equation*}
\begin{split}
&(x_3-x_2)(x_3-x_1)(x_2-x_1)h_{m}(x_1,x_2,x_3)
\\ &=(x_3-x_2)(x_3-x_1)(x_2-x_1)\big((x_1+x_2)h_{m-1}(x_1,x_2,x_3)-x_1x_2h_{m-2}(x_1,x_2,x_3)+x_{3}^m \big)
\\ &=(x_{3}-x_{2})x_{1}^{m+2}+(x_{1}-x_{3})x_{2}^{m+2}+(x_{2}-x_{1})x_{3}^{m+2} .
\end{split}
\end{equation*}
The last equality is by \eqnref{eqn3} and \eqnref{eqn4}.
\end{proof}
Before answering question \ref{ques factor} for $n \geq 4$ variables, we answer question
\ref{ques triangle} as an application of \propref{prop3var}.

Let $p(t)$, $x_1$, $x_2$, $x_3$, and the triangle $ABC$ be as given in question \ref{ques triangle}.
Then the signed area $A(x_1 ,x_2,x_3)$ of the triangle $ABC$ is given by $\frac{1}{2} \det(\mathbf{M}) $,
where
$ \displaystyle \mathbf{M}=\begin{pmatrix} 1&x_1&p(x_1)\\
               1&x_2&p(x_2)\\
               1&x_3&p(x_3)
\end{pmatrix}.$
Thus the area of $ABC$ is $\left| A(x_1 ,x_2,x_3) \right|$ and we have
\begin{equation}\label{eqn5}
A(x_1,x_2,x_3)=\frac{1}{2}\big((x_3-x_2)p(x_1)-(x_3-x_1)p(x_2)+(x_2-x_1)p(x_3)\big)
\end{equation}
(See \cite[Page 1]{MPF}, for the relation between this area and the convexity of $p(t)$.)
We note that  $\det(\mathbf{M})$ is a particular example of an alternant, a definition first used
by Sylvester \cite[Page 322]{SM} for determinants of matrices such that $i-$th row of $\mathbf{M}$ are
functions of variable $x_{i}$ and same functions are used for each row.
As mentioned in the article \cite{SM}, differences of the variables divide $A(x_1,x_2,x_3)$. By the following
theorem, which uses \propref{prop3var}, we express $A(x_1,x_2,x_3)$ in a factorized form.
\begin{theorem} \label{Main Theorem}
Let the polynomial $p(t)=a_0+a_1 t+ \dots + a_N t^N$ and the triangle $ABC$ be as above, then
\begin{equation*}
 \left| A(x_1 ,x_2,x_3) \right| =\frac{1}{2} \left| x_1-x_2 \right| \left| x_1-x_3 \right| \left| x_2-x_3 \right|
\left| \sum_{k=2}^{N} a_{k}h_{k-2}(x_1,x_2,x_3) \right| .
\end{equation*}
\end{theorem}
\begin{proof}
The proof is given by induction on $N$, the degree of the polynomial $p(t)$. Without loss of generality, suppose that
$x_1 \leq x_2 \leq x_3$. It is enough to prove that
\begin{equation}\label{eqn6}
A(x_1,x_2,x_3) =\frac{1}{2} (x_2-x_1)(x_3-x_1)(x_3-x_2)\sum_{k=2}^{N} a_{k}h_{k-2}(x_1,x_2,x_3).
\end{equation}
If $N=2$, $A(x_1,x_2,x_3)$ simplifies to $\frac{1}{2}(x_2-x_1)(x_3-x_1)(x_3-x_2)a_{2}$, which is the desired form
since $h_{0}(x_1,x_2,x_3)=1$.

Assume that for some $N \geq 2$, \eqnref{eqn6} holds for all polynomials of degree at most $N$.
Then, given a polynomial $P(t)$ of degree $N+1$, we can write $P(t)=a_{N+1}t^{N+1}+p(t)$ where $a_{N+1}\not=0$ and
$p(t)$ has degree at most $N$. We then have
\begin{equation*}
\begin{split}
A(x_1,x_2,x_3)&=\frac{1}{2}\big((x_3-x_2)(a_{N+1}x_{1}^{N+1}+p(x_1))-(x_3-x_1)(a_{N+1}x_{2}^{N+1}+p(x_2))
\\ & \qquad +(x_2-x_1)(a_{N+1}x_{3}^{N+1}+p(x_3))\big), \qquad \text{by \eqnref{eqn5}}.
\\ &=\frac{1}{2}\big((x_3-x_2)p(x_1)-(x_3-x_1)p(x_2)+(x_2-x_1)p(x_3)\big)
\\ & \qquad +\frac{1}{2}a_{N+1}\big((x_3-x_2)x_{1}^{N+1}-(x_3-x_1)x_{2}^{N+1}
+(x_2-x_1)x_{3}^{N+1}\big)\\
&=\frac{1}{2} (x_2-x_1)(x_3-x_1)(x_3-x_2) \sum_{k=2}^{N} a_{k}h_{k-2}(x_1,x_2,x_3)+\frac{1}{2}a_{N+1}\big((x_3-x_2)x_{1}^{N+1}
\\ & \qquad -(x_3-x_1)x_{2}^{N+1}+(x_2-x_1)x_{3}^{N+1}\big), \qquad \text{by induction}.
\end{split}
\end{equation*}
\begin{equation*}
\begin{split}
\\ &=\frac{1}{2} (x_2-x_1)(x_3-x_1)(x_3-x_2) \sum_{k=2}^{N+1} a_{k}h_{k-2}(x_1,x_2,x_3),  \qquad \text{by \propref{prop3var}}.
\end{split}
\end{equation*}
\end{proof}
We answer question \ref{ques poly} by the following well-known identity for complete symmetric homogeneous polynomials (see \cite[Ex 7.4, p. 450 and p. 490]{RS}). We provide an elementary proof which relies on induction arguments.
\begin{theorem}\label{thm symmetric poly}
For $1 \leq m \in \NN$, $2 \leq n \in \NN$, and distinct variables $x_1, x_2, \dots x_n$,
\begin{equation*}
\sum_{i=1}^{n}\frac{x_{i}^{n+m-1}}
{\prod_{\substack{j=1 \\ j\not=i}}^{n}(x_i-x_j)}=h_{m}(x_1, x_2, \dots, x_n).
\end{equation*}
\end{theorem}
\begin{proof}
We obtain the result by induction on $K=n+m \geq 3$. The cases $K=3$ and $K=4$ follows from
\eqnref{eqn2} and \eqnref{prop3var}, respectively.
Suppose $K>4$.
\begin{equation*}
\begin{split}
h_{m}(x_1, x_2, \dots, x_n) &=h_{m}(x_1, x_2, \dots, x_{n-1})+x_n \cdot h_{m-1}(x_1, x_2, \dots, x_n), \quad \text{by grouping}\\
&=\sum_{i=1}^{n-1}\frac{x_{i}^{n+m-2}}
{\prod_{\substack{j=1 \\ j\not=i}}^{n-1}(x_i-x_j)}+x_n \sum_{i=1}^{n}\frac{x_{i}^{n+m-2}}
{\prod_{\substack{j=1 \\ j\not=i}}^{n}(x_i-x_j)}, \quad \text{by induction}\\
&=\sum_{i=1}^{n-1}\frac{x_{i}^{n+m-2}(x_i-x_n)}
{\prod_{\substack{j=1 \\ j\not=i}}^{n}(x_i-x_j)}+ \sum_{i=1}^{n-1} x_n \frac{x_{i}^{n+m-2}}
{\prod_{\substack{j=1 \\ j\not=i}}^{n}(x_i-x_j)} + \frac{x_n^{n+m-1}}{\prod_{\substack{j=1 \\ j\not=i}}^{n}(x_i-x_j)}.
\end{split}
\end{equation*}
This gives the result.
\end{proof}
If $n<m$, \thmref{thm symmetric poly} can be used to compute $h_{m}(x_1, x_2, \dots, x_n)$ effectively.
A result closely related to \thmref{thm symmetric poly} is the following theorem:
\begin{theorem}\label{thm symmetric}
Let $2 \leq n \in \NN$, and let $x_1, x_2, \ldots, x_n$ be distinct variables. Then for any $0 \leq k \leq n-2$, we have
$\displaystyle \sum_{i=1}^{n}\frac{x_{i}^{k}}
{\prod_{\substack{j=1 \\ j\not=i}}^{n}(x_i-x_j)}=0.$
\end{theorem}
\begin{proof}
Case $1$: $k=n-2$.
For $n=2$, we have $\frac{1}{x_1-x_2}+\frac{1}{x_2-x_1}=0$. Let $n>2$, and
let $k=n-2$. Lagrange interpolation formula applied to the polynomial $q(x_n)=-x_n^k$ gives
$$-x_n^k=\sum_{i=1}^{n-1}\frac{-x_i^k \cdot \prod_{\substack{j=1 \\ j\not=i}}^{n-1}(x_n-x_j)}
{\prod_{\substack{j=1 \\ j\not=i}}^{n-1}(x_i-x_j)}=
\sum_{i=1}^{n-1}\frac{x_i^k \cdot \prod_{\substack{j=1 }}^{n-1}(x_n-x_j)}
{\prod_{\substack{j=1 \\ j\not=i}}^{n}(x_i-x_j)}
=\prod_{\substack{j=1 }}^{n-1}(x_n-x_j) \sum_{i=1}^{n-1}\frac{x_i^k}
{\prod_{\substack{j=1 \\ j\not=i}}^{n}(x_i-x_j)}
.$$ This completes the proof in this case.

Case $2$: $0 \leq k=n-3$. We have

$\displaystyle \sum_{i=1}^{n-1}\frac{x_i^{n-3}}{\prod_{\substack{j=1 \\ j\not=i}}^{n-1}(x_i-x_j)}=
\sum_{i=1}^{n-1}\frac{x_i^{n-3}(x_i-x_n)}{\prod_{\substack{j=1 \\ j\not=i}}^{n}(x_i-x_j)}
=\sum_{i=1}^{n-1}\frac{x_i^{n-2}}{\prod_{\substack{j=1 \\ j\not=i}}^{n}(x_i-x_j)}-x_n
\sum_{i=1}^{n-1}\frac{x_i^{n-3}}{\prod_{\substack{j=1 \\ j\not=i}}^{n}(x_i-x_j)}.
$
Equivalently,
$ \displaystyle \sum_{i=1}^{n-1}\frac{x_i^{n-3}}{\prod_{\substack{j=1 \\ j\not=i}}^{n-1}(x_i-x_j)}=
\sum_{i=1}^{n}\frac{x_i^{n-2}}{\prod_{\substack{j=1 \\ j\not=i}}^{n}(x_i-x_j)}
-x_n
\sum_{i=1}^{n}\frac{x_i^{n-3}}{\prod_{\substack{j=1 \\ j\not=i}}^{n}(x_i-x_j)}.
$
Note that case $1$ with $n-1$ and $n$ implies
$ \displaystyle \sum_{i=1}^{n-1}\frac{x_i^{n-3}}{\prod_{\substack{j=1 \\ j\not=i}}^{n-1}(x_i-x_j)}=0
=\sum_{i=1}^{n}\frac{x_i^{n-2}}{\prod_{\substack{j=1 \\ j\not=i}}^{n}(x_i-x_j)}.
$
This completes the proof in this case.

Following the strategy of case $2$ successively, one can prove the result for each $k$ with $0 \leq k \leq n-4$.
\end{proof}
We now generalize \propref{prop3var} and \thmref{Main Theorem} to an arbitrary number of variables.
We first note that discriminants of the variables appear in both \eqnref{eqn2} and \propref{prop3var}.

A difference product of the variables $x_1, x_2, \dots, x_n$
is the following Vandermonde determinant:
$\displaystyle d(x_1, x_2, \dots, x_n)=\det{(\mathbf{D})}, \text{ where } \mathbf{D}=\begin{pmatrix} 1&x_1&x_1^2&\dots&x_{1}^{n-1}\\
               1&x_2&x_2^2&\dots&x_{2}^{n-1}\\
               \hdotsfor[2]{5}\\
               1&x_n&x_n^2&\dots&x_{n}^{n-1}
\end{pmatrix}_{n \times n}.$
%
It was shown by Cauchy in 1812 that
\begin{equation}\label{eqn disc1}
d(x_1, x_2, \dots, x_n)=\prod_{1\leq i<j \leq n}(x_j-x_i) .
\end{equation}
The following theorem is the answer for question \ref{ques factor} for arbitrary
number of variables:
\begin{theorem}\label{genprop3var}
For $0 \leq m \in \NN$, $2 \leq n \in \NN$, and distinct variables $x_1, x_2, \dots ,x_n$, we have
\begin{equation*}
\sum_{k=1}^{n}(-1)^{n-k}x_{k}^{n+m-1}d(x_1,x_2,\dots,x_{k-1},x_{k+1},\dots, x_n)=d(x_1, x_2, \dots, x_n)h_{m}(x_1, x_2, \dots, x_n).
\end{equation*}
\end{theorem}
\begin{proof}
Using \eqnref{eqn disc1},
$
\displaystyle d(x_1,x_2,\dots,x_{k-1},x_{k+1},\dots, x_n)=\prod_{\substack{1 \leq i<j \leq n \\ i\not=k \not=j}}^{n}(x_j-x_i).
$
Then the result follows from \thmref{thm symmetric poly}.
\end{proof}
For $0 \leq m \in \NN$, $2 \leq n \in \NN$, let
$ \displaystyle \mathbf{A}=\begin{pmatrix} 1&x_1&x_1^2&\dots&x_{1}^{n-2} \, \, x_1^{n+m-1}\\
               1&x_2&x_2^2&\dots&x_{2}^{n-2} \, \, x_2^{n+m-1}\\
               \hdotsfor[2]{5}\\
               1&x_n&x_n^2&\dots&x_{n}^{n-2} \, \, x_n^{n+m-1}
\end{pmatrix}_{n \times n} .$
Then \thmref{genprop3var} can be expressed equivalently as follows:
\begin{equation}\label{eqn determinantal}
\det(\mathbf{A})= d(x_1, x_2, \dots, x_n)h_{m}(x_1,x_2, \dots, x_n).
\end{equation}
We now generalize \thmref{Main Theorem} to an arbitrary number of variables.
Let the polynomial $p(t)$ of degree $N$ be defined as earlier. We set
$ \displaystyle \mathbf{M_{n}}=\begin{pmatrix} 1&x_1&x_1^2&\dots&x_{1}^{n-2} \, \, p(x_1)\\
               1&x_2&x_2^2&\dots&x_{2}^{n-2} \, \, p(x_2)\\
               \hdotsfor[2]{5}\\
               1&x_n&x_n^2&\dots&x_{n}^{n-2} \, \, p(x_n)
\end{pmatrix}_{n \times n}.$
\begin{remark}\label{rem volume}
The volume of  $n$-parallelpipeds generated by the column vectors of $\mathbf{M_n}$ is
$\left| \det(\mathbf{M_{n}}) \right|$, and the volumes of the $n$-simplices generated by the column vectors of $\mathbf{M_n}$ is
$\frac{1}{n!} \left| \det(\mathbf{M_{n}}) \right|$.
\end{remark}
\begin{theorem}\label{genmainthhm}
Let $2\leq n\in \NN$. If $N \leq n-2$, then $\det(\mathbf{M_{n}})=0$. If $n-1\leq N \in \NN$,
\begin{equation*}
\det(\mathbf{M_{n}})=d(x_1, x_2, \dots, x_n)\sum_{k=n-1}^{N} a_{k}h_{k-n+1}(x_1,x_2, \dots, x_n) .
\end{equation*}
\end{theorem}
\begin{proof}
The proof is given by induction on $N$, the degree of the polynomial $p(t)$ (as in the proof of \thmref{Main Theorem}),
and by using \eqnref{eqn determinantal}.
The details are left as an exercise to the reader.
\end{proof}
\begin{acknowledgement}
I would like to thank Dr. Robert Rumely for useful discussions concerning this paper.
\end{acknowledgement}

\end{document}